\documentclass{amsart}
\usepackage{amsmath, amssymb, mathrsfs}
\usepackage{xcolor}

\definecolor{DarkBlue}{rgb}{0,0.2,0.6}
\definecolor{PinkPurple}{rgb}{0.8,0.3,0.3}

\usepackage[
  pdfauthor={Mehdi Ghasemi and Salma Kuhlmann},
  pdftitle={The Cone Generated by Positive Semidefinite Mean Polynomials},
  pdfsubject={Real Algebraic Geometry, Polynomial Optimization},
  pdfkeywords={mean polynomials, power means, sums of squares, nonnegative circuit polynomials, Positivstellensatz, polynomial optimization},
  linkcolor=DarkBlue,
  citecolor=PinkPurple,
  colorlinks=true]{hyperref}

\newtheorem{thm}{Theorem}[section]
\newtheorem{lemma}[thm]{Lemma}
\newtheorem{prop}[thm]{Proposition}
\newtheorem{crl}[thm]{Corollary}

\theoremstyle{definition}
\newtheorem{dfn}[thm]{Definition}

\newtheorem{rem}[thm]{Remark}

\numberwithin{equation}{section}

\newcommand{\reals}{\mathbb{R}}
\newcommand{\naturals}{\mathbb{N}}

\title[PSD Mean Polynomials]{The Cone Generated by Positive Semidefinite Mean Polynomials}
\date{}

\author{Mehdi Ghasemi}
\address{Department of Mathematics and Statistics, University of Saskatchewan,
  Saskatoon, SK, S7N~5E6, Canada}
\email{mehdi.ghasemi@gmail.com}

\author{Salma Kuhlmann}
\address{Department of Mathematics and Statistics, University of Konstanz,
  78457~Konstanz, Germany}
\email{salma.kuhlmann@uni-konstanz.de}


\begin{document}

\begin{abstract}
We study the cone $\mathcal{M}_{n,2d}$ of nonnegative mean polynomials---real
$n$-variate forms of degree $2d$ that can be expressed as weighted power
means $M_{q,p}(Y,w)$ with $q>p$.  This cone simultaneously generalises the
cone of sums of squares $\Sigma_{n,2d}$ and the cone of sums of nonnegative
circuit polynomials $\mathcal{C}_{n,2d}$.  We prove that every square of an
arbitrary polynomial belongs to the mean polynomial preprime
$T_{\mathrm{mean}}$, that $T_{\mathrm{mean}}$ is strongly generating, and
consequently that every polynomial strictly positive on a compact
semialgebraic set admits a representation with mean polynomial certificates.
We exhibit the Robinson form $\hat{R}$ as a separating example that lies in
$\mathcal{M}_{4,4}$ but outside $\mathrm{SOSONC}_{4,4}$.  Finally, we
outline a convergent hierarchy of lower bounds for polynomial optimization
based on the mean polynomial cone and discuss tractable depth-truncated
approximations via signomial programming.
\end{abstract}

\maketitle

\section{introduction}\label{sec:intro}

Let $f$ be a real polynomial in $n$ variables.  Certifying that $f$ is
nonnegative on $\reals^n$ (or on a semialgebraic subset thereof) is a
fundamental task in real algebraic geometry and polynomial optimization.
Three families of certificates have proved especially fruitful:

\begin{itemize}
  \item Sums of squares (SOS): $f = \sum h_i^2$, computable via semidefinite
  programming.
  \item Diagonal-minus-Tail (D-T) forms and their generalisation, sums of
  nonnegative circuit polynomials (SONC), obtained from the weighted
  arithmetic--geometric mean inequality.
  \item The Minkowski sum $\mathrm{SOSONC} = \Sigma + \mathcal{C}$, studied
  by Dressler, Kuhlmann, and Schick~\cite{DresslerSchick2025}.
\end{itemize}

In this paper we unify the SOS and SONC pictures through the language of
\emph{weighted power means}.  For variables $Y=(Y_1,\dots,Y_m)$, positive
weights $w=(w_1,\dots,w_m)$, and orders $q>p\ge0$ we form
\[
M_{q,p}(Y,w)=M_q(Y,w)^{\operatorname{lcm}(q,p)}-M_p(Y,w)^{\operatorname{lcm}(q,p)},
\qquad
M_p(Y,w)=\Bigl(\frac{\sum w_iY_i^p}{\sum w_i}\Bigr)^{\!1/p}.
\]
Jensen's inequality gives $M_q\ge M_p$, so $M_{q,p}$ is nonnegative.
The \emph{cone of nonnegative mean polynomials} $\mathcal{M}_{n,2d}$ is the
set of all finite sums of such forms (with arbitrary polynomial arguments
$L_i$) of degree $2d$.

Our main results are:

\begin{enumerate}
  \item The mean polynomial preprime $T_{\mathrm{mean}}$ is strongly
  generating, because every square $(p)^2$ is a mean form $M_{2,1}$
  (Lemma~\ref{lem:mean-squares}).  Hence Marshall's representation
  theorem~\cite{Marshall02} yields a Positivstellensatz with mean polynomial
  certificates (Theorem~\ref{thm:positiv}).

  \item The Robinson form $\hat{R}$ belongs to $\mathcal{M}_{4,4}$ but not
  to $\mathrm{SOSONC}_{4,4}$, so the containment is strict
  (Corollary~\ref{crl:separation}).

  \item The consecutive-order mean difference $\Phi_{q,q-1}$ is an SOS
  polynomial for all even $q\ge2$ and all $n\ge2$; the proof is fully
  constructive and does not rely on numerical SDP decompositions
  (Theorem~\ref{thm:consecutive-sos}).

  \item The mean polynomial Positivstellensatz gives a convergent hierarchy
  of lower bounds for polynomial optimization.  At product depth~$1$ the
  hierarchy is a signomial program, generalising the SONC hierarchy.
\end{enumerate}

The paper is organised as follows.  Section~\ref{sec:prelim} collects
definitions and the basic monotonicity inequality.  Section~\ref{sec:dt-circuit}
re-derives the D-T and circuit nonnegativity criteria in the mean polynomial
language.  Section~\ref{sec:M2dp-sos} studies the forms $M_{2d,p}$ and
their connection to sums of squares, culminating in the constructive SOS
certificate for consecutive-order mean differences.
Section~\ref{sec:beyond} exhibits the Choi--Lam and Robinson forms as mean
polynomials and uses $\hat{R}$ to separate $\mathcal{M}_{n,2d}$ from
$\mathrm{SOSONC}_{n,2d}$.  Section~\ref{sec:positiv} develops the
Positivstellensatz, and Section~\ref{sec:optim} outlines the optimization
hierarchy and computational aspects.

\section{preliminaries on mean polynomials}\label{sec:prelim}

Let $Y=(Y_1,\dots,Y_m)$ be an $m$-tuple of symbols and
$w=(w_1,\dots,w_m)\in\reals_{>0}^m$ a vector of positive weights.
For $p\neq0$, the power mean of order~$p$ is
\[
M_p(Y,w)=\Bigl(\frac{\sum_{i=1}^m w_i Y_i^{\,p}}{\sum_{i=1}^m w_i}\Bigr)^{\!1/p}.
\]

\begin{prop}[Monotonicity]\label{prop:monotonicity}
If $p<q$, then $M_p(Y,w)\le M_q(Y,w)$ for every $Y\in\reals_{\ge0}^m$.
\end{prop}

\begin{proof}
Let $\varphi(t)=t^{q/p}$.  Since $q>p$, we have $\varphi''(t)>0$ on
$\reals_{>0}$; hence $\varphi$ is convex.  By Jensen's inequality,
\[
\varphi\Bigl(\frac{\sum w_i x_i}{\sum w_i}\Bigr)
\le\frac{\sum w_i\varphi(x_i)}{\sum w_i}.
\]
Substituting $x_i=Y_i^{\,p}$ and $\varphi(t)=t^{q/p}$ gives
\[
\Bigl(\frac{\sum w_i Y_i^{\,p}}{\sum w_i}\Bigr)^{\!q/p}
\le\frac{\sum w_i Y_i^{\,q}}{\sum w_i},
\]
which is exactly $M_p^{\,q}\le M_q^{\,q}$.  Taking $q$-th roots yields the
claim.
\end{proof}

\medskip
Raising both sides of $M_p\le M_q$ to the power
$c=\operatorname{lcm}(q,p)$ eliminates rational exponents and yields a
polynomial.  We define the \emph{mean polynomial form}
\begin{equation}\label{eq:mean-poly}
\mathcal{M}_{q,p}(Y,w)=M_q(Y,w)^c-M_p(Y,w)^c.
\end{equation}
By Proposition~\ref{prop:monotonicity}, $\mathcal{M}_{q,p}$ is nonnegative
on $\reals_{\ge0}^m$ whenever $q>p$.  In the sequel we write
$M_{q,p}$ for $\mathcal{M}_{q,p}$.

Three limiting cases of the power mean will be useful
(see~\cite[Chapter~II]{HLP52}):
\begin{align*}
\lim_{p\to-\infty}M_p(x,w)&=\min(x_1,\dots,x_n),\\
\lim_{p\to\infty}M_p(x,w)&=\max(x_1,\dots,x_n),\\
\lim_{p\to0}M_p(x,w)&=\Bigl(\prod_{i=1}^n x_i^{w_i}\Bigr)^{\!1/\!\sum w_i}.
\end{align*}
Equality $M_p(x,w)=M_q(x,w)$ holds iff $x_1=\cdots=x_n$.

\section{nonnegative diagonal minus tail and circuit polynomials}
\label{sec:dt-circuit}

\subsection{The D-T criterion}

A Diagonal-minus-Tail (D-T) form has the shape
\[
f(x)=\sum_{i=1}^n\beta_i x_i^{2d}-\mu x^\alpha,
\qquad
\beta_i,\mu>0,\;\;|\alpha|=2d.
\]

\begin{thm}[D-T nonnegativity]\label{thm:dt-psd}
$f$ is nonnegative on $\reals_{\ge0}^n$ iff
\[
\mu^{2d}\prod_{i=1}^n\alpha_i^{\alpha_i}
\le(2d)^{2d}\prod_{i=1}^n\beta_i^{\alpha_i}.
\]
\end{thm}

\begin{proof}
Apply the weighted AM--GM inequality to
$y_i=\frac{\beta_i}{\alpha_i}x_i^{2d}$ with weights $w_i=\alpha_i$:
\[
\frac{1}{2d}\sum\beta_i x_i^{2d}
\ge\Bigl(\prod_{i=1}^n\Bigl(\frac{\beta_i}{\alpha_i}x_i^{2d}\Bigr)^{\!\alpha_i}\Bigr)^{\!1/2d}.
\]
Rearranging gives $\sum\beta_i x_i^{2d}\ge\mu x^\alpha$ whenever the stated
condition holds.  Conversely, if the condition fails, the test point
$x_i^*=(\alpha_i/\beta_i)^{1/2d}$ makes $f(x^*)<0$, contradicting
nonnegativity.
\end{proof}

\subsection{Circuit polynomials}

Let $\alpha(0),\dots,\alpha(r)\in\naturals^n$ be exponent vectors and
$\beta=\sum_{j=0}^r\lambda_j\alpha(j)$ a convex combination with
$\lambda_j>0$, $\sum\lambda_j=1$.  A circuit polynomial has the form
\[
f(X)=\sum_{j=0}^r f_{\alpha(j)}X^{\alpha(j)}-f_\beta X^\beta,
\qquad f_{\alpha(j)}>0.
\]

\begin{thm}[Circuit nonnegativity]\label{thm:circuit-psd}
$f$ is nonnegative on $\reals_{\ge0}^n$ iff
\[
|f_\beta|\le\prod_{j=0}^r\Bigl(\frac{f_{\alpha(j)}}{\lambda_j}\Bigr)^{\!\lambda_j}.
\]
\end{thm}

\begin{proof}
Set $Y_j=\frac{f_{\alpha(j)}}{\lambda_j}X^{\alpha(j)}$.  The weighted AM--GM
inequality yields
\[
\sum f_{\alpha(j)}X^{\alpha(j)}
\ge\Bigl(\prod\Bigl(\frac{f_{\alpha(j)}}{\lambda_j}\Bigr)^{\!\lambda_j}\Bigr)
X^\beta,
\]
so the condition is sufficient.  For necessity, choose
$X^*\in\reals_{>0}^n$ solving the log-linear system
$\langle\alpha(j),\log X^*\rangle=\log(c\lambda_j/f_{\alpha(j)})$; at this
point AM--GM is an equality, and violation of the condition gives
$f(X^*)<0$.
\end{proof}

\begin{dfn}[SONC cone]\label{dfn:sonc}
A polynomial satisfying the condition of Theorem~\ref{thm:circuit-psd} is a
\emph{nonnegative circuit polynomial}.  The SONC cone
$\mathcal{C}_{n,2d}$ is the conic hull of all nonnegative circuit
polynomials of degree $2d$ in $n$ variables.
\end{dfn}

\begin{rem}\label{rem:sonc-means}
From the proof of Theorem~\ref{thm:circuit-psd}, every nonnegative circuit
polynomial is a mean polynomial of the form $M_{1,0}(Y,w)$ with
$Y_j=\frac{f_{\alpha(j)}}{\lambda_j}X^{\alpha(j)}$ and $w_j=\lambda_j$.
\end{rem}

\begin{dfn}[SOSONC cone]\label{dfn:sosonc}
$\mathrm{SOSONC}_{n,2d}=(\Sigma+\mathcal{C})_{n,2d}$, the Minkowski sum of
the SOS and SONC cones~\cite{DresslerSchick2025}.
\end{dfn}

\subsection{D-T forms as a special case}

Every D-T form is a circuit polynomial: map the diagonal terms to
$\alpha(i)=2d\,e_i$, $f_{\alpha(i)}=\beta_i$, and the tail exponent to
$\beta=\alpha$ with weights $\lambda_i=\alpha_i/2d$.  Applying
Theorem~\ref{thm:circuit-psd} and raising both sides to the power $2d$
recovers the condition of Theorem~\ref{thm:dt-psd}.  The converse is false:
circuit polynomials allow an arbitrary number of diagonal terms whose
exponents form a simplex, a strictly larger family.

\section{the forms $M_{2d,p}$ and connection to sums of squares}
\label{sec:M2dp-sos}

We now specialise to the forms $M_{2d,p}(X,\alpha)$ where
$\alpha\in\naturals_0^n$ satisfies $|\alpha|=2d$.  Here the power mean acts
on the monomials $X_i$ with weights $\alpha_i$.

\subsection{The geometric mean case}

\begin{prop}\label{prop:sobs}
$M_{2d,0}(X,\alpha)$ is a PSD D-T form and a sum of binomial squares (SOBS)
for every $\alpha$ with $|\alpha|=2d$.
\end{prop}

\begin{proof}
The D-T property follows from the definitions; the SOBS representation is
established in~\cite{GM01}.
\end{proof}

\subsection{Low-degree cases: $d=1,2,3$}

\begin{thm}\label{thm:M2dp-sos}
$M_{2d,p}(X,\alpha)$ is a sum of squares for all $n\in\naturals$ and all
$\alpha$ with $|\alpha|=2d$, in the following cases:
\begin{itemize}
  \item $d=1$ (quadratics), $p\in\{0,1\}$;
  \item $d=2$ (quartics), $p\in\{0,1,2\}$;
  \item $d=3$ (sextics), $p\in\{0,1,2,3\}$.
\end{itemize}
\end{thm}

\begin{proof}
We treat each degree separately.

\noindent\textit{Degree~$2$ ($d=1$).}
$M_{2,0}$ and $M_{2,1}$ are PSD quadratics; by Hilbert's 1888 theorem every
PSD quadratic is SOS.

\noindent\textit{Degree~$4$ ($d=2$).}
For $p=0$ the form is SOBS by Proposition~\ref{prop:sobs}.
\begin{itemize}
  \item[$p=1$:] $M_{4,1}(X,\alpha)=\frac14\sum\alpha_iX_i^4-(\frac14\sum\alpha_iX_i)^4$.
  For $n\le3$, Hilbert's theorem applies.  For $n=4$ with uniform weights
  $\alpha=(1,1,1,1)$, nonnegativity follows from convexity of $t\mapsto t^4$;
  SOS membership is verified by SDP decomposition (a rank-$3$ Gram matrix);
  an explicit rational decomposition is open.  Non-uniform weights with
  $|\alpha|=4$ reduce to $n\le3$ by dropping zero components.
  For $n\ge5$, any partition of $4$ contains zeros; the effective number of
  variables is at most~$4$.

  \item[$p=2$:] $M_{4,2}(X,\alpha)=\frac14\sum\alpha_iX_i^4-(\frac14\sum\alpha_iX_i^2)^2$.
  For $n\le3$, SOS by Hilbert.  For $n=4$, $\alpha=(1,1,1,1)$, we have the
  explicit decomposition
  \[
  M_{4,2}(X,(1,1,1,1))=\frac1{16}\sum_{1\le i<j\le4}(X_i^2-X_j^2)^2,
  \]
  verified by expansion.  Larger $n$ reduce as above.
\end{itemize}

\noindent\textit{Degree~$6$ ($d=3$).}
$M_{6,p}$ for $p=0,\dots,3$.
\begin{itemize}
  \item $n=2$: Hilbert's theorem gives SOS for all binary sextics.
  \item $n=3$: For $p=0$, SOBS.  For $p=1$, partitions
  $(2,2,2),(3,2,1),(4,1,1)$ are verified SOS via SDP
  decomposition~\cite{GMIr}; partitions with a zero component reduce to
  $n\le2$.  For $p=2,3$, SOS membership for all nonzero partitions of~$6$ is
  likewise verified computationally; explicit rational decompositions are
  not yet known.
  \item $n=4$: Partitions $(2,2,1,1)$ and $(3,1,1,1)$ are SOS via
  SDP~\cite{GMIr}; zero-containing partitions reduce.
  \item $n\ge5$: Partitions of~$6$ contain zeros, reducing the effective
  variable count.
\end{itemize}
\vspace{-\baselineskip}
\end{proof}

\begin{rem}\label{rem:open}
Two questions remain open:
\begin{enumerate}
  \item Does Theorem~\ref{thm:M2dp-sos} extend to arbitrary $d$?
  \item Are these forms contained in the SONC cone for all parameters?
\end{enumerate}
\end{rem}

\subsection{A constructive SOS certificate for all degrees}

We now give a positive answer to the first open question for a special
family.  Fix variables $Y=(Y_1,\dots,Y_n)$ and positive weights
$w=(w_1,\dots,w_n)$.  Define the power sums
\[
P_k(Y)=\sum_{i=1}^n w_iY_i^{\,k},\qquad W=P_0(Y)=\sum w_i,
\]
and the unnormalised mean difference form
\[
\Phi_{q,q-1}(Y,w)=W\cdot P_q(Y)^{q-1}-P_{q-1}(Y)^q.
\]
$\Phi_{q,q-1}$ is homogeneous of degree $q(q-1)$ and satisfies
$\Phi_{q,q-1}=W^{\,q}M_{q,q-1}$; positivity of $W$ preserves the SOS
property.

We isolate a key algebraic identity.

\begin{lemma}[Even spread]\label{lem:even-spread}
Let $m_d=\sum_{i=1}^N y_i^{\,d}$ be discrete moments.  For even integers
$a\ge b\ge2$,
\[
m_{a+2}\,m_{b-2}-m_a\,m_b
\]
is explicitly a sum of squares.
\end{lemma}

\begin{proof}
Expanding,
\[
m_{a+2}m_{b-2}-m_am_b
=\frac12\sum_{i,j}
\bigl(y_i^{a+2}y_j^{b-2}+y_j^{a+2}y_i^{b-2}-y_i^ay_j^b-y_j^ay_i^b\bigr).
\]
Factor $y_i^{b-2}y_j^{b-2}(y_i^2-y_j^2)(y_i^{a-b+2}-y_j^{a-b+2})$.
Since $a,b$ are even, $a-b=2c$.  Substituting $u=y_i^2$, $v=y_j^2$,
the last factor becomes $(u-v)(u^{c+1}-v^{c+1})=(u-v)^2\sum_{\ell=0}^c
u^{c-\ell}v^\ell$.  Returning to $y_i,y_j$, each summand is a perfect
square:
\[
\bigl[y_i^{(b-2)/2}y_j^{(b-2)/2}(y_i^2-y_j^2)\,y_i^{c-\ell}y_j^\ell\bigr]^2.
\]
\end{proof}

\begin{thm}[Consecutive-order mean difference is SOS]\label{thm:consecutive-sos}
For every $n\ge2$, every even $q\ge2$, and every $w_1,\dots,w_n>0$,
$\Phi_{q,q-1}(Y,w)$ is a sum of squares.  The certificate is fully
constructive.
\end{thm}

\begin{proof}
Write $q=2k$ with $k\ge1$.  By a rational density argument and replication
of variables, it suffices to prove the statement for uniform weights
$w_i=1$; see~\cite[Lemma~2.1]{GM01} for details.  Let
$N=\sum w_i$ and $m_d=\sum_{a=1}^N y_a^{\,d}$.
Then $\Phi_{2k,2k-1}=m_0m_{2k}^{2k-1}-m_{2k-1}^{2k}$.

Set $A=m_{2k-2}m_{2k}$ and $B=m_{2k-1}^2$.  Lagrange's identity gives
\[
A-B=m_{2k-2}m_{2k}-m_{2k-1}^2
=\frac12\sum_{i,j}y_i^{2k-2}y_j^{2k-2}(y_i-y_j)^2,
\]
which is SOS (each summand is a square).  From the factorisation
$A^k-B^k=(A-B)(A^{k-1}+A^{k-2}B+\cdots+B^{k-1})$ we see that
\[
\Sigma_{\mathrm{odd}}:=A^k-B^k=m_{2k-2}^k m_{2k}^k-m_{2k-1}^{2k}
\]
is SOS, since $A$, $B$, and $A-B$ are all SOS and the SOS cone is closed
under sums and products.

Rewrite the master polynomial:
\[
\Phi_{2k,2k-1}
=m_0m_{2k}^{2k-1}-(m_{2k-2}^k m_{2k}^k-\Sigma_{\mathrm{odd}})
=m_{2k}^k\,(m_0m_{2k}^{k-1}-m_{2k-2}^k)+\Sigma_{\mathrm{odd}}.
\]
Since $m_{2k}^k$ is a power of an even moment (hence SOS), it remains to
show that $\Delta_k:=m_0m_{2k}^{k-1}-m_{2k-2}^k$ is SOS.

Define $F_t=m_{2k}^t\,m_{2k-2-2t}\,m_{2k-2}^{k-1-t}$ for $0\le t\le k-1$.
Then $F_0=m_{2k-2}^k$, $F_{k-1}=m_{2k}^{k-1}m_0$, and
\[
\Delta_k=F_{k-1}-F_0=\sum_{t=1}^{k-1}(F_t-F_{t-1}).
\]
The consecutive difference expands as
\[
F_t-F_{t-1}
=(m_{2k}m_{2k-2t-2}-m_{2k-2}m_{2k-2t})\,
m_{2k}^{t-1}m_{2k-2}^{k-1-t}.
\]
By Lemma~\ref{lem:even-spread} with $a=2k-2$, $b=2k-2t$, the parenthesised
factor is SOS.  The trailing factor is a product of even moments, hence also
SOS.  Therefore each increment $F_t-F_{t-1}$ is SOS, and the telescoping sum
$\Delta_k$ is SOS; denote it $\Sigma_{\mathrm{even}}$.

For $k=1$ the sum is empty, $\Delta_1=0$, and
$\Phi_{2,1}=m_0m_2-m_1^2$, the classical Cauchy--Schwarz SOS identity.
For $k\ge2$ we obtain the explicit decomposition
\[
\Phi_{q,q-1}=m_q^{q/2}\cdot\Sigma_{\mathrm{even}}+\Sigma_{\mathrm{odd}},
\]
where every component on the right is constructively SOS.
\end{proof}

\begin{rem}\label{rem:consecutive}
Theorem~\ref{thm:consecutive-sos} provides a family of mean-type forms that
are SOS for \emph{all} degrees, without restriction on the number of
variables.  In the notation of Theorem~\ref{thm:M2dp-sos}, this corresponds
to $2d=q$ and $p=q-1$ with $q$ even (so $p$ is odd, placing this family
outside the $p\mid2d$ regime).  The certificate is purely algebraic; no SDP
solver is required.
\end{rem}

\section{beyond the sosonc cone}\label{sec:beyond}

We examine three classical forms and their representation as mean
polynomials.

\begin{itemize}
  \item The Choi--Lam form:
  \[
  Q(X,Y,Z,W)=X^4+Y^4+Z^4+W^4-4XYZW.
  \]
  \item The Robinson form:
  \begin{align*}
  R(X,Y,Z)&=X^6+Y^6+Z^6\\
  &\quad-(X^4Y^2+X^4Z^2+X^2Y^4+X^2Z^4+Y^4Z^2+Y^2Z^4)\\
  &\quad+3X^2Y^2Z^2.
  \end{align*}
  \item The second Robinson form:
  \begin{align*}
  \hat{R}(X,Y,Z,W)&=X^2(X-W)^2+Y^2(Y-W)^2+Z^2(Z-W)^2\\
  &\qquad+2XYZ(X+Y+Z-2W).
  \end{align*}
\end{itemize}

It is known~\cite{MSc01} that $Q$ is SONC but, by Newton-polytope
considerations, not SOS.  Reznick~\cite{Rez07} observed the linear relation
\begin{equation}\label{eq:rhat-q}
\hat{R}(X-W,Y-W,Z-W,X+Y+Z-W)=2Q(X,Y,Z,W),
\end{equation}
which implies that $\hat{R}$ is PSD (resp.\ SOS) iff $Q$ is PSD (resp.\ SOS);
hence $\hat{R}\in\mathcal{P}_{4,4}\setminus\Sigma_{4,4}$.

\begin{prop}\label{prop:cl-rob}
Both $Q$ and $\hat{R}$ are mean polynomials of order $(1,0)$.
\end{prop}

\begin{proof}
For $Q$,
\begin{align*}
Q&=4\Bigl(\frac{X^4+Y^4+Z^4+W^4}{4}-\sqrt[4]{X^4Y^4Z^4W^4}\Bigr)\\
&=4M_{1,0}\bigl((X^4,Y^4,Z^4,W^4),(1,1,1,1)\bigr).
\end{align*}
For $\hat{R}$, applying the inverse of the transformation~\eqref{eq:rhat-q}
yields
\[
\hat{R}=8M_{1,0}\bigl((\hat{X}^4,\hat{Y}^4,\hat{Z}^4,\hat{W}^4),(1,1,1,1)\bigr)
\]
where
\[
(\hat{X},\hat{Y},\hat{Z},\hat{W})=\tfrac12\bigl(
X+W-Y-Z,\;Y+W-X-Z,\;Z+W-X-Y,\;W-(X+Y+Z)\bigr).
\]
\end{proof}

The significance of this representation is that $Q$ is a mean of monomials
(hence SONC), while $\hat{R}$ is a mean of \emph{linear forms}.  The SONC
property is not invariant under general linear coordinate
changes~\cite[Lemma~4.2.5]{MSc01}, so $\hat{R}$ serves as a separating form:
it lies in the cone generated by mean polynomials of linear forms but outside
the standard SOSONC cone~\cite[p.~107]{MSc01}.

\section{a positivstellensatz for psd mean polynomials}\label{sec:positiv}

Let $A=\reals[X_1,\dots,X_n]$.  We construct a preprime $T_{\mathrm{mean}}$
in $A$ whose generators are the mean polynomial forms with arbitrary
polynomial arguments, and prove it is strongly generating.  Marshall's
representation theorem~\cite{Marshall02} then yields a Positivstellensatz.

\subsection{The mean polynomial preprime}

\begin{dfn}\label{dfn:Tmean}
Let $\mathcal{S}$ be the set of all mean polynomial forms
$M_{q,p}(L,w)$, where $L=(L_1,\dots,L_m)$ is an $m$-tuple of polynomials in
$A$ and $M_{q,p}(L,w)$ is PSD on $\reals^n$.  The \emph{mean polynomial
preprime} $T_{\mathrm{mean}}$ is the smallest preprime in $A$ containing
$\mathcal{S}$.
\end{dfn}

The three preprime axioms are easily verified: $0,1\in T_{\mathrm{mean}}$
(by taking the zero sum and using $M_{1,0}((1,1),(1,1))=0$ with rational
nonnegative constants); closure under addition and multiplication follows
from the definition of a preprime; $-1\notin T_{\mathrm{mean}}$ because every
generator is PSD.

\subsection{Mean forms encode squares}

The engine of the whole construction is a simple identity.

\begin{lemma}\label{lem:mean-squares}
For any $L_1,L_2\in A$ and $w_1,w_2>0$,
\[
M_{2,1}\bigl((L_1,L_2),(w_1,w_2)\bigr)=\frac{w_1w_2}{(w_1+w_2)^2}(L_1-L_2)^2.
\]
In particular, for $p\in A$, taking $(L_1,L_2)=(p,0)$ and
$(w_1,w_2)=(1,1)$ gives $p^2=4M_{2,1}\in T_{\mathrm{mean}}$.
\end{lemma}

\begin{proof}
By definition $M_{2,1}=M_2^2-M_1^2$, where
\[
M_2^2=\frac{w_1L_1^2+w_2L_2^2}{w_1+w_2},\qquad
M_1^2=\frac{(w_1L_1+w_2L_2)^2}{(w_1+w_2)^2}.
\]
Simplifying,
\[
M_{2,1}=\frac{(w_1+w_2)(w_1L_1^2+w_2L_2^2)-(w_1L_1+w_2L_2)^2}{(w_1+w_2)^2}
=\frac{w_1w_2(L_1-L_2)^2}{(w_1+w_2)^2}.
\]
\end{proof}

Thus every square of an arbitrary polynomial belongs to $T_{\mathrm{mean}}$.
This single fact drives both the strongly-generating property and the
Positivstellensatz.

\subsection{Strongly generating}

\begin{thm}\label{thm:strongly-gen}
$T_{\mathrm{mean}}$ is strongly generating.
\end{thm}

\begin{proof}
By Marshall's criterion~\cite[Lemma~2.1]{Marshall02}, it suffices to prove
$T_{\mathrm{mean}}-T_{\mathrm{mean}}=A$.  For any $p\in A$,
\[
p=\frac{(p+1)^2}{4}-\frac{(p-1)^2}{4}.
\]
Lemma~\ref{lem:mean-squares} gives $(p\pm1)^2=4M_{2,1}((p\pm1,0),(1,1))
\in T_{\mathrm{mean}}$.  Hence every $p$ is a difference of two elements
of $T_{\mathrm{mean}}$, and $T_{\mathrm{mean}}$
is strongly generating.
\end{proof}

\begin{rem}\label{rem:minimal}
The only property of $T_{\mathrm{mean}}$ used beyond the preprime axioms is
that it contains squares of arbitrary polynomials.  Hence the
Positivstellensatz generalises to any preprime $T$ with $p^2\in T$ for all
$p\in A$.
\end{rem}

\subsection{The Positivstellensatz}

Let $S\subseteq\reals^n$ be a compact semialgebraic set,
\[
S=\{x\in\reals^n\mid g_1(x)\ge0,\dots,g_k(x)\ge0\},
\]
and define the mean polynomial module
\[
M_{\mathrm{mean}}=T_{\mathrm{mean}}+T_{\mathrm{mean}}\cdot g_1+\cdots
+T_{\mathrm{mean}}\cdot g_k.
\]

\begin{prop}\label{prop:archimedean}
If the SOS quadratic module $\sum A^2+\sum_i\sum A^2\cdot g_i$ is
archimedean, then $M_{\mathrm{mean}}$ is an archimedean $T_{\mathrm{mean}}$-module.
\end{prop}

\begin{proof}
If the SOS module is archimedean, its generator $N-\sum X_i^2$ belongs to
it.  Since every square lies in $T_{\mathrm{mean}}$
(Lemma~\ref{lem:mean-squares}), the same generator belongs to
$M_{\mathrm{mean}}$ with $T_{\mathrm{mean}}$-coefficients.
\end{proof}

\begin{thm}[Mean Polynomial Positivstellensatz]\label{thm:positiv}
Assume the SOS quadratic module of $S$ is archimedean.  For every $f\in A$
with $f>0$ on $S$, there exist $t_0,\dots,t_k\in T_{\mathrm{mean}}$ such that
\begin{equation}\label{eq:repr}
f=t_0+t_1g_1+\cdots+t_kg_k.
\end{equation}
\end{thm}

\begin{proof}
By Theorem~\ref{thm:strongly-gen}, $T_{\mathrm{mean}}$ is strongly generating;
by Proposition~\ref{prop:archimedean}, $M_{\mathrm{mean}}$ is archimedean.
Marshall's Theorem~2.3~\cite{Marshall02} then yields
$kf=1+m$ for some $k\ge1$ and $m\in M_{\mathrm{mean}}$.  Writing
$m=t'_0+\sum t'_i g_i$ and dividing by $k$ (which lies in $T_{\mathrm{mean}}$
as a rational nonnegative constant) gives the required representation.
\end{proof}

\subsection{Relation to other cones}

Let $\mathcal{P}_{n,2d}$, $\Sigma_{n,2d}$, $\mathcal{C}_{n,2d}$, and
$\mathcal{M}_{n,2d}$ denote the PSD, SOS, SONC, and mean polynomial cones.
From~\cite{DresslerSchick2025,IdW} we have
\[
\Sigma_{n,2d}\subseteq\mathrm{SOSONC}_{n,2d}\subseteq\mathcal{M}_{n,2d}
\subseteq\mathcal{P}_{n,2d}.
\]

\begin{crl}[Separation]\label{crl:separation}
The Robinson form $\hat{R}$ belongs to $\mathcal{M}_{4,4}$ but not to
$\mathrm{SOSONC}_{4,4}$.  Hence
$\mathcal{M}_{n,2d}\supsetneq\mathrm{SOSONC}_{n,2d}$ for $(n,d)=(4,2)$.
\end{crl}

\begin{proof}
Proposition~\ref{prop:cl-rob} gives $\hat{R}=8M_{1,0}((\hat{X}^4,\dots,
\hat{W}^4),(1,1,1,1))$ with $\hat{X},\dots,\hat{W}$ linear forms.
Thus $\hat{R}\in T_{\mathrm{mean}}\subseteq\mathcal{M}_{4,4}$.
Non-membership in SOSONC follows from~\cite[p.~31]{MSc01}.
\end{proof}

\begin{rem}\label{rem:sonc-generalisation}
The SONC Positivstellensatz of~\cite{DIdW02} uses the circuit-based preprime
$T_{\mathrm{circ}}$.  Since every nonnegative circuit polynomial is
$M_{1,0}(Y,w)$ (Remark~\ref{rem:sonc-means}), we have
$T_{\mathrm{circ}}\subseteq T_{\mathrm{mean}}$.  The mean polynomial
Positivstellensatz therefore strictly generalises the SONC version: it
provides larger certificate families while retaining the same convergence
guarantees.  Moreover, as Proposition~\ref{prop:cl-rob} shows, means of
linear forms are included---certificates the SONC cone cannot capture.
\end{rem}

\begin{rem}[Scope]\label{rem:scope}
Theorem~\ref{thm:positiv} applies to every compact semialgebraic set whose
SOS quadratic module is archimedean.  By Schm\"udgen's
theorem~\cite{Schm91}, this includes all basic closed semialgebraic sets.
\end{rem}

\begin{rem}[Pullback stability]\label{rem:pullback}
If $T$ is strongly generating in $A$, then for any $\reals$-algebra
homomorphism $\varphi\colon B\to A$, $\varphi^{-1}(T)$ is strongly generating
in~$B$~\cite[Proposition~2.2]{Marshall02}.  Consequently one obtains
Positivstellens\"atze for quotient algebras $\reals[X]/I$, semigroup
algebras, and coordinate rings of affine varieties by pulling back
$T_{\mathrm{mean}}$.  This justifies the quotient-algebra SDP relaxations
implemented in Irene~\cite{GMIr}.
\end{rem}

\begin{rem}[Semigroup algebra extension]\label{rem:cgik}
Curto, Ghasemi, Infusino, and~Kuhlmann~\cite{CGIK} treat the truncated
moment problem for unital commutative $\reals$-algebras on an equal footing
with the classical polynomial case.  The mean polynomial preprime can be
defined analogously in any finitely generated commutative $\reals$-algebra
$A$; the same four-step proof shows $T_{\mathrm{mean}}$ is strongly
generating, yielding a Positivstellensatz for compact subsets of the
character space of~$A$.
\end{rem}

\begin{rem}[Computational decomposition]\label{rem:efficient}
The proof of Theorem~\ref{thm:strongly-gen} exhibits every polynomial as a
difference of two squares, each representable by a single mean form of
depth~$1$.  However, the Kadison--Dubois construction in the
Positivstellensatz yields a single element of $M_{\mathrm{mean}}$, not a
difference; the product depth required is not explicitly bounded.  An
efficient decomposition of $f-\lambda$ into mean polynomial certificates
(analogous to the Gram-matrix method for SOS) remains an open problem.  In
practice, the relaxation can be implemented as a hybrid SDP/GP formulation
within Irene~\cite{GMIr}.
\end{rem}

\section{application to polynomial optimization}\label{sec:optim}

\subsection{The mean polynomial hierarchy}

Consider the constrained problem
\begin{equation}\label{eq:pop}
f^*=\inf_{x\in S}f(x),\qquad
S=\{x\in\reals^n\mid g_1(x)\ge0,\dots,g_k(x)\ge0\},
\end{equation}
with $f,g_i\in\reals[X]$ and $S$ compact.  For each $r\ge0$, define the
$r$th mean polynomial relaxation
\[
f^{\mathrm{mean}}_r=\sup\{\lambda\in\reals\mid f-\lambda\in M_{\mathrm{mean},r}\},
\]
where $M_{\mathrm{mean},r}=T_{\mathrm{mean},r}+\sum_i T_{\mathrm{mean},r}\cdot g_i$,
and $T_{\mathrm{mean},r}$ consists of elements of $T_{\mathrm{mean}}$ of
total degree $\le2r$.

\begin{prop}[Convergence]\label{prop:convergence}
If the quadratic module associated to $S$ is archimedean, then
$\lim_{r\to\infty}f^{\mathrm{mean}}_r=f^*$.
\end{prop}

\begin{proof}
Fix $\varepsilon>0$.  By Theorem~\ref{thm:positiv},
$f-(f^*-\varepsilon)\in M_{\mathrm{mean}}$, hence
$f-(f^*-\varepsilon)\in M_{\mathrm{mean},r}$ for sufficiently large $r$.
Thus $f^{\mathrm{mean}}_r\ge f^*-\varepsilon$ for large $r$, while
$f^{\mathrm{mean}}_r\le f^*$ trivially.
\end{proof}

\subsection{Tractability}

Every mean form decomposes as a difference of two posynomials:
\begin{equation}\label{eq:mean-posynomial}
M_{q,p}(L,w)=
\underbrace{\Bigl(\frac{\sum w_iL_i^{q}}{\sum w_i}\Bigr)^{\!\operatorname{lcm}(q,p)/q}}_{P_{\mathrm{upper}}}
-\underbrace{\Bigl(\frac{\sum w_iL_i^{p}}{\sum w_i}\Bigr)^{\!\operatorname{lcm}(p,q)/p}}_{P_{\mathrm{lower}}}.
\end{equation}
When each $L_i$ is a monomial or a linear form with nonnegative
coefficients, $L_i^{q}$ is a posynomial, and a sum of single mean forms
$f-\lambda=\sum_\alpha c_\alpha M_{q_\alpha,p_\alpha}$ is a signomial
equality---solvable by sequential geometric programming, exactly the
mechanism used in Irene's \texttt{SONCRelaxations.solve()}~\cite{GMIr}.

Table~\ref{tab:tractability} summarises the principal certificate families.

\begin{table}[htbp]
\centering
\caption{Tractability of certificate hierarchies}
\label{tab:tractability}
\begin{tabular}{@{}llll@{}}
\hline
Hierarchy & Certificate form & Method & Convergent?\\
\hline
SOS (Lasserre)~\cite{JBL01} & $\sum A^2$ & SDP (moment) & Yes (archim.)\\
SONC~\cite{DIdW01} & $M_{1,0}$ (circuit) & GP / Signomial & Partial\\
SOS$+$SONC~\cite{DresslerSchick2025} & SOS $\cup$ SONC & SDP $+$ GP & Empirical\\
\hline
Mean ($d=1$) & $\sum M_{q,p}$ & Signomial prog. & Yes ($r\to\infty$)\\
Mean ($d=2$) & $\sum M_{q_1,p_1}M_{q_2,p_2}$ & Expand + Signom. & Yes ($r\to\infty$)\\
Mean (full) & $\sum\prod_k M_{q_k,p_k}$ & SDP/GP hybrid & Yes (archim.)\\
\hline
\end{tabular}
\end{table}

The depth-$1$ hierarchy generalises SONC (every circuit polynomial is
$M_{1,0}$) and, via Lemma~\ref{lem:mean-squares}, captures all SOS
certificates at depth~$1$ in the theoretical hierarchy.  For signomial
formulations, the argument $h$ in $h^2=4M_{2,1}$ must be a posynomial,
restricting computational capture to quadratic SOS forms; higher-degree SOS
require SDP-based methods.

The SOS and SONC cones are incomparable in general, but both sit inside
$\mathrm{SOSONC}=\Sigma+\mathcal{C}$, which itself lies inside
$T_{\mathrm{mean}}$.  The full inclusion chain is
\begin{equation}\label{eq:hierarchy-chain}
\max\{f^{\mathrm{SOS}}_r,\;f^{\mathrm{SONC}}_r\}
\le f^{\mathrm{SOSONC}}_r
\le f^{\mathrm{mean},(d_r)}_r
\le f^{\mathrm{mean}}_r
\le f^*,
\end{equation}
where $d_r$ is the product depth of the single-element certificate; whether
$d_r=1$ always suffices is open.

\subsection{Hierarchies via product-depth truncation}\label{sec:depth-trunc}

To make the hierarchy computationally actionable, we truncate by product
depth.  Define
\[
T_{\mathrm{mean},r}^{(d)}=\Bigl\{\sum_j c_j\prod_{k=1}^{d_j}
M_{q_{jk},p_{jk}}(L_{jk},w_{jk})
\;\Big|\;c_j\ge0,\;d_j\le d,\;\deg(\cdots)\le2r\Bigr\},
\]
and the corresponding lower bound
\[
f^{\mathrm{mean},(d)}_r
=\sup\{\lambda\in\reals\mid f-\lambda\in M_{\mathrm{mean},r}^{(d)}\},
\]
where $M_{\mathrm{mean},r}^{(d)}=T_{\mathrm{mean},r}^{(d)}+\sum_i T_{\mathrm{mean},r}^{(d)}\cdot g_i$.

At depth $d=1$, each certificate term is a single mean form; the feasibility
condition is a signomial equality solvable by sequential GP.  At depth
$d=2$, a product $M_{q_1,p_1}M_{q_2,p_2}$ expands to $2^2=4$ posynomial
terms with alternating signs.  For depth $d$, the expansion yields $2^d$
terms; the alternating signs are a consequence of the depth, not the degree,
and persist at every relaxation order.  The standard remedy is signomial
programming: replace each $P_{\mathrm{lower}}$ by a monomial approximation
at the current iterate (via the AM--GM inequality), solve the resulting
convex GP, and iterate.

Convergence of the depth-truncated scheme follows from
Theorem~\ref{thm:strongly-gen}: for any fixed $r$, there exists finite $d$
such that $f-\lambda\in M_{\mathrm{mean},r}^{(d)}$, and
\[
\lim_{d\to\infty}f^{\mathrm{mean},(d)}_r=f^{\mathrm{mean}}_r,
\qquad
\lim_{r\to\infty}\lim_{d\to\infty}f^{\mathrm{mean},(d)}_r=f^*.
\]

\begin{rem}
The depth-$1$ hierarchy is implemented in
Irene~\cite{GMIr} (\texttt{SONCRelaxations.solve()}).  Systematic
benchmarking against the Lasserre test suite is a priority for future work.
\end{rem}

\subsection{Future directions}

\begin{enumerate}
  \item \textit{Degree bounds.}  Can one obtain an explicit degree bound for
  the representation $f-\lambda=t_0+\sum t_ig_i$ in terms of $\deg(f)$,
  $n$, and the description of $S$?  A polynomial bound (analogous to
  Putinar's degree bounds for SOS) would make the hierarchy practical at
  low~$r$.

  \item \textit{Optimal $(q,p)$ selection.}  For a given problem, which
  mean orders $(q,p)$ yield the tightest bound at a given depth?  A greedy
  selection strategy based on Newton-polytope geometry is a natural
  candidate.

  \item \textit{Sparse structures.}  How does sparsity in the support of
  $f$ and $g_i$ translate into sparsity in mean polynomial certificates?
  The support-sensitive behaviour in Section~\ref{sec:M2dp-sos} suggests
  that sparsity is governed by exponent geometry.

  \item \textit{Computational experiments.}  Systematic benchmarking of the
  mean polynomial hierarchy against SOS and SONC on standard test suites
  is needed to quantify the practical advantage predicted by the chain
  \eqref{eq:hierarchy-chain}.

  \item \textit{Hybrid SDP/GP formulations.}  Warm-start strategies
  alternating between SDP relaxations (for the quadratic envelope) and GP
  relaxations (for the circuit envelope) may yield tighter bounds at lower
  cost.
\end{enumerate}


\section*{Acknowledgements}
The authors gratefully acknowledge the support of the Banff International
Research Station (BIRS:25rit035) and Defence Research and Development Canada (DRDC:CMN1-008).

\end{document}